\documentclass[a4paper,10pt,leqno]{amsart}
\usepackage[cp1250]{inputenc}
\usepackage{ifthen}
\usepackage{amssymb}
\usepackage{amsmath}
\usepackage{amsfonts}
\usepackage{latexsym}

\numberwithin{equation}{section}
\newtheorem{theorem}{Theorem}[section]
\theoremstyle{plain}

\newtheorem{corollary}[theorem]{Corollary}

\newtheorem{lemma}{Lemma}[section]

\numberwithin{equation}{section}
\input{tcilatex}

\begin{document}
\title[Second Hankel determinant for bi-starlike and bi-convex functions of
order $\beta $]{Second Hankel determinant for bi-starlike and bi-convex
functions of order $\beta $}
\author{Erhan Deniz}
\author{Murat Ça\u{g}lar}
\address{Department of Mathematics, Faculty of Science and Letters, Kafkas
University, Kars, Turkey.}
\email{edeniz36@gmail.com (Erhan Deniz), mcaglar25@gmail.com (Murat Ça\u{g}%
lar)}
\author{Halit Orhan}
\address{Department of Mathematics, Faculty of Science, Ataturk University,
Erzurum, 25240, Turkey.}
\email{orhanhalit607@gmail.com (Halit Orhan)}
\subjclass[2000]{ Primary 30C45, 30C50; Secondary 30C80.}
\keywords{Bi-univalent functions, bi-starlike functions of order $\beta $,
bi-convex functions of order $\beta $, second Hankel determinant. \\
\textit{Corresponding author}. edeniz36@gmail.com (Erhan Deniz)}

\begin{abstract}
In the present investigation the authors obtain upper bounds for the second
Hankel determinant $H_{2}(2)$ of the classes bi-starlike and bi-convex
functions of order $\beta $, represented by $S_{\sigma }^{\ast }(\beta )\;$%
and $K_{\sigma }(\beta )$, respectively. In particular, the estimates for
the second Hankel determinant $H_{2}(2)$ of bi-starlike and bi-convex
functions which are important subclasses of bi-univalent functions are
pointed out.
\end{abstract}

\maketitle

\section{Introduction and definitions}

Let $\mathcal{A}$ denote the family of functions $f$\ analytic in the open
unit disk $\mathcal{U}=\left\{ {z\in 
\mathbb{C}
:\left\vert z\right\vert <1}\right\} $ of the form 
\begin{equation}
f(z)=z+\sum\limits_{n=2}^{\infty }{a_{n}z^{n}.}  \label{eq1}
\end{equation}%
Let $\mathcal{S}$ denote the class of all functions in $\mathcal{A}$ which
are univalent in $\mathcal{U}$. The Koebe one-quarter theorem (see \cite%
{Duren 1}) ensures that the image of $\mathcal{U}$ under every $f\in 
\mathcal{S}$ contain a disk of radius $1\diagup 4.$ So, every $f\in \mathcal{%
S}$ has an inverse function $f^{-1}$ satisfying $f^{-1}(f(z))=z$ $\left(
z\in \mathcal{U}\right) $ and%
\begin{equation*}
f(f^{-1}(w))=w\text{ \ \ }\left( \left\vert w\right\vert <r_{0}(f);\text{ }%
r_{0}(f)\geq 1\diagup 4\right)
\end{equation*}%
where $%
f^{-1}(w)=w-a_{2}w^{2}+(2a_{2}^{2}-a_{3})w^{3}-(5a_{2}^{3}-5a_{2}a_{3}+a_{4})w^{4}+.... 
$

A function $f\in \mathcal{A}$ is said to be bi-univalent in $\mathcal{U}$ if
both $f(z)$ and $f^{-1}(z)$ are univalent in $\mathcal{U}.$ Let $\sigma $
denote the class of bi-univalent functions in $\mathcal{U}$ given by (\ref%
{eq1}).

Two of the most famous subclasses of univalent functions are the class $%
\mathcal{S}^{\ast }(\beta )$ of starlike functions of order $\beta $ and the
class $\mathcal{K}(\beta )\;$of convex functions of order $\beta $. By
definition, we have%
\begin{equation*}
\mathcal{S}^{\ast }(\beta )=\left\{ f\in \mathcal{S}:\Re \left( \frac{%
zf^{\prime }(z)}{f(z)}\right) >\beta ;\;z\in \mathcal{U};\;0\leq \beta
<1\right\}
\end{equation*}%
and%
\begin{equation*}
\mathcal{K}(\beta )=\left\{ f\in \mathcal{S}:\Re \left( 1+\frac{zf^{\prime
\prime }(z)}{f^{\prime }(z)}\right) >\beta ;\;z\in \mathcal{U};\;0\leq \beta
<1\right\} .
\end{equation*}%
The classes consisting of starlike and convex functions are usually denoted
by $\mathcal{S}^{\ast }=\mathcal{S}^{\ast }(0)$ and $\mathcal{K}=\mathcal{K}%
(0)$, respectively.

For $0\leq \beta <1,$ a function $f\in \sigma $ is in the class $\mathcal{S}%
_{\sigma }^{\ast }(\beta )$ of \textit{bi-starlike functions of order} $%
\beta ,$ or $\mathcal{K}_{\sigma }(\beta )$ of \textit{bi-convex functions
of order} $\beta $ if both $f$ and its inverse map $f^{-1}$ are,
respectively, starlike or convex of order $\beta .$ These classes were
introduced by Brannan and Taha \cite{Bran} in 1985. Especially the classes $%
\mathcal{S}_{\sigma }^{\ast }(0)=\mathcal{S}_{\sigma }^{\ast }$ and $%
\mathcal{K}_{\sigma }(0)=\mathcal{K}_{\sigma }$ are \textit{bi-starlike} and 
\textit{bi-convex functions}, respectively. In 1967, Lewin \cite{Lew} showed
that for every functions $f\in \sigma $ of the form (\ref{eq1}), the second
coefficient of $f$ satisfy the inequality $|a_{2}|<1.51$. In 1967, Brannan
and Clunie \cite{Br1} conjectured that $|a_{2}|\leq \sqrt{2}$ for $f\in
\sigma .$ Later, Netanyahu \cite{Ne} proved that $\max_{f\in \sigma
}|a_{2}|=4/3$. \ In 1985, Kedzierawski \cite{Ked} proved Brannan and
Clunie's conjecture for $f\in \mathcal{S}_{\sigma }^{\ast }$. In 1985, Tan 
\cite{Tan} obtained the bound for $a_{2}$ namely $|a_{2}|<1.485$ which is
the best known estimate for functions in the class $\sigma .$ Brannan and
Taha \cite{Bran} obtained estimates on the initial coefficients $\left\vert
a_{2}\right\vert $ and $\left\vert a_{3}\right\vert $ for functions in the
classes $\mathcal{S}_{\sigma }^{\ast }(\beta )$ and $\mathcal{K}_{\sigma
}(\beta )$. Recently, Deniz \cite{Deniz} and Kumar et al. \cite{Ku} both
extended and improved the results of Brannan and Taha \cite{Bran} by
generalizing their classes using subordination. The problem of estimating
coefficients $\left\vert a_{n}\right\vert $, $n\geq 2$ is still open.
However, a lot of results for $\left\vert a_{2}\right\vert $, $\left\vert
a_{3}\right\vert $ and $\left\vert a_{4}\right\vert $ were proved for some
subclasses of $\sigma $ (see \cite{Bul}, \cite{CaOrNi}, \cite{Fra}, \cite%
{Jahan}, \cite{Or2}, \cite{Sri}, \cite{Sri2}, \cite{Qing1}, \cite{Qing2}).
Unfortunatelly, none of them are not sharp.

One of the important tools in the theory of univalent functions is Hankel
Determinants which are utility, for example, in showing that a function of
bounded characteristic in $\mathcal{U}$, i.e., a function which is a ratio
of two bounded analytic functions, with its Laurent series around the origin
having integral coefficients, is rational \cite{Cantor}. The Hankel
determinants \cite{No-Tho}$\;H_{q}(n)\;(n=1,2,...,\;q=1,2,...)$ of the
function $f$\ are defined by%
\begin{equation*}
H_{q}(n)=\left[ 
\begin{array}{cccc}
a_{n} & a_{n+1} & ... & a_{n+q-1} \\ 
a_{n+1} & a_{n+2} & ... & a_{n+q} \\ 
\vdots & \vdots &  & \vdots \\ 
a_{n+q-1} & a_{n+q} & ... & a_{n+2q-2}%
\end{array}%
\right] \text{ \ \ }(a_{1}=1).
\end{equation*}

This determinant was discussed by several authors with $q=2$. For example,
we can know that the functional $H_{2}(1)=a_{3}-a_{2}^{2}\;$is known as the
Fekete-Szegö functional and they consider the further generalized functional 
$a_{3}-\mu a_{2}^{2}\;$where $\mu $ is some real number (see, \cite{Fekete}%
). In 1969, Keogh and Merkes \cite{ke} proved the Fekete-Szegö problem for
the classes $\mathcal{S}^{\ast }$ and $\mathcal{K}$. Someone can see the
Fekete-Szegö problem for the classes $\mathcal{S}^{\ast }(\beta )$ and $%
\mathcal{K}(\beta )$ at special cases in the paper of Orhan \textit{et.al. }%
\cite{Or}. On the other hand, very recently Zaprawa \cite{Za}, \cite{Za1}
have studied on Fekete-Szegö problem for some classes of bi-univalent
functions. In special cases, he gave Fekete-Szegö problem for the classes $%
\mathcal{S}_{\sigma }^{\ast }(\beta )$ and $\mathcal{K}_{\sigma }(\beta )$.
In 2014, Zaprawa \cite{Za} proved the following resuts for $\mu \in 
\mathbb{R}
,$%
\begin{equation*}
f\in \mathcal{S}_{\sigma }^{\ast }(\beta )\Rightarrow \left\vert a_{3}-\mu
a_{2}^{2}\right\vert \leq \left\{ 
\begin{array}{cc}
1-\beta ; & \frac{1}{2}\leq \mu \leq \frac{3}{2} \\ 
2(1-\beta )\left\vert \mu -1\right\vert ; & \mu \geq \frac{3}{2}\text{ and }%
\mu \leq \frac{1}{2}%
\end{array}%
\right.
\end{equation*}%
and%
\begin{equation*}
f\in \mathcal{K}_{\sigma }(\beta )\Rightarrow \left\vert a_{3}-\mu
a_{2}^{2}\right\vert \leq \left\{ 
\begin{array}{cc}
\frac{1-\beta }{3}; & \frac{2}{3}\leq \mu \leq \frac{4}{3} \\ 
(1-\beta )\left\vert \mu -1\right\vert ; & \mu \geq \frac{4}{3}\text{ and }%
\mu \leq \frac{2}{3}%
\end{array}%
\right. .
\end{equation*}

The second Hankel determinant $H_{2}(2)$ is given by $%
H_{2}(2)=a_{2}a_{4}-a_{3}^{2}.$ The bounds for the second Hankel determinant 
$H_{2}(2)$ obtained for the classes $\mathcal{S}^{\ast }$ and $\mathcal{K}$
in \cite{Ja}. Recently, Lee et al. \cite{Lee} established the sharp bound to 
$\left\vert H_{2}(2)\right\vert $ by generalizing their classes using
subordination. In their paper, one can find the sharp bound to $\left\vert
H_{2}(2)\right\vert $ for the functions in the classes $\mathcal{S}^{\ast
}(\beta )$ and $\mathcal{K}(\beta )$.

In this paper, we seek upper bound for the functional $%
H_{2}(2)=a_{2}a_{4}-a_{3}^{2}\;$for functions $f\;$belonging to the classes $%
\mathcal{S}_{\sigma }^{\ast }(\beta )\;$and $\mathcal{K}_{\sigma }(\beta )$.

Let $\mathcal{P}$ be the class of functions with positive real part
consisting of all analytic functions $\mathcal{P}:\mathcal{U\rightarrow 
\mathbb{C}
}$ satisfying $p(0)=1$ and $\Re p(z)>0$.

To establish our main results, we shall require the following lemmas.

\begin{lemma}
\label{l1}\cite{Pom} If the function $p\in \mathcal{P}$ is given by the
series 
\begin{equation}
p(z)=1+c_{1}z+c_{2}z^{2}+...  \label{eq1*}
\end{equation}%
then the sharp estimate $\left\vert c_{k}\right\vert \leq 2$ $(k=1,2,...)$
holds.
\end{lemma}

\begin{lemma}
\label{l2}\cite{Grenan} If the function $p\in \mathcal{P}$ is given by the
series (\ref{eq1*}), then%
\begin{eqnarray}
2c_{2} &=&c_{1}^{2}+x(4-c_{1}^{2})  \label{eq2} \\
4c_{3}
&=&c_{1}^{3}+2(4-c_{1}^{2})c_{1}x-c_{1}(4-c_{1}^{2})x^{2}+2(4-c_{1}^{2})%
\left( 1-\left\vert x\right\vert ^{2}\right) z,  \label{eq3}
\end{eqnarray}%
for some $x,$ $z$ with $\left\vert x\right\vert \leq 1$ and $\left\vert
z\right\vert \leq 1.$
\end{lemma}

\section{Main results}

Our first main result for the class $\mathcal{S}_{\sigma }^{\ast }(\beta )$
as follows:

\begin{theorem}
\label{t1} Let $f(z)$ given by (\ref{eq1}) be in the class $\mathcal{S}%
_{\sigma }^{\ast }(\beta ),$ $0\leq \beta <1.$ Then%
\begin{equation}
\left\vert a_{2}a_{4}-a_{3}^{2}\right\vert \leq \left\{ 
\begin{array}{cc}
\frac{4}{3}\left( 1-\beta \right) ^{2}\left( 4\beta ^{2}-8\beta +5\right) ,
& \beta \in \left[ 0,\frac{29-\sqrt{137}}{32}\right] \\ 
\left( 1-\beta \right) ^{2}\left( \frac{13\beta ^{2}-14\beta -7}{16\beta
^{2}-26\beta +5}\right) , & \beta \in \left( \frac{29-\sqrt{137}}{32}%
,1\right) .%
\end{array}%
\right.  \label{eq5}
\end{equation}
\end{theorem}

\begin{proof}
Let $f\in \mathcal{S}_{\sigma }^{\ast }(\beta )\;$and $g=f^{-1}.$\ Then%
\begin{equation}
\frac{zf^{\prime }(z)}{f(z)}=\beta +(1-\beta )p(z)\;\text{and }\frac{%
wg^{\prime }(w)}{g(w)}=\beta +(1-\beta )q(w)  \label{eq6}
\end{equation}%
where $p(z)=1+c_{1}z+c_{2}z^{2}+...$ and $q(w)=1+d_{1}w+d_{2}w^{2}+...$ in $%
\mathcal{P}$.

Comparing coefficients in (\ref{eq6}), we have%
\begin{eqnarray}
a_{2} &=&(1-\beta )c_{1},  \label{eq8} \\
2a_{3}-a_{2}^{2} &=&(1-\beta )c_{2},  \label{eq9} \\
3a_{4}-3a_{3}a_{2}+a_{2}^{3} &=&(1-\beta )c_{3}  \label{eq10}
\end{eqnarray}%
and%
\begin{eqnarray}
-a_{2} &=&(1-\beta )d_{1},  \label{eq8*} \\
3a_{2}^{2}-2a_{3} &=&(1-\beta )d_{2},  \label{eq9*} \\
-10a_{2}^{3}+12a_{3}a_{2}-3a_{4} &=&(1-\beta )d_{3}.  \label{eq10*}
\end{eqnarray}%
From (\ref{eq8}) and (\ref{eq8*}), we arrive at 
\begin{equation}
c_{1}=-d_{1}  \label{eq11}
\end{equation}%
and%
\begin{equation}
a_{2}=(1-\beta )c_{1}.  \label{eq12}
\end{equation}

Now, from (\ref{eq9}), (\ref{eq9*}) and (\ref{eq12}), we get that%
\begin{equation}
a_{3}=\left( 1-\beta \right) ^{2}c_{1}^{2}+\frac{\left( 1-\beta \right) }{4}%
\left( c_{2}-d_{2}\right) .  \label{eq13}
\end{equation}

Also, from (\ref{eq10}) and (\ref{eq10*}), we find that%
\begin{equation}
a_{4}=\frac{2}{3}\left( 1-\beta \right) ^{3}c_{1}^{3}+\frac{5}{8}\left(
1-\beta \right) ^{2}c_{1}\left( c_{2}-d_{2}\right) +\frac{1}{6}(1-\beta
)\left( c_{3}-d_{3}\right) .  \label{eq14}
\end{equation}%
Thus, we can easily establish that%
\begin{eqnarray}
&&\left\vert a_{2}a_{4}-a_{3}^{2}\right\vert =\left\vert -\frac{1}{3}\left(
1-\beta \right) ^{4}c_{1}^{4}+\frac{1}{8}\left( 1-\beta \right)
^{3}c_{1}^{2}\left( c_{2}-d_{2}\right) \right.  \notag \\
&&\text{ \ \ \ \ \ \ \ \ \ \ \ \ \ \ }\left. +\frac{1}{6}\left( 1-\beta
\right) ^{2}c_{1}\left( c_{3}-d_{3}\right) -\frac{1}{16}\left( 1-\beta
\right) ^{2}\left( c_{2}-d_{2}\right) ^{2}\right\vert .  \label{eq15}
\end{eqnarray}%
According to Lemma \ref{l2} and (\ref{eq11}), we write%
\begin{equation}
\left. 
\begin{array}{c}
2c_{2}=c_{1}^{2}+x(4-c_{1}^{2}) \\ 
2d_{2}=d_{1}^{2}+x(4-d_{1}^{2})%
\end{array}%
\right\} \Longrightarrow c_{2}-d_{2}=\frac{4-c_{1}^{2}}{2}(x-y)  \label{eq16}
\end{equation}%
and%
\begin{eqnarray*}
4c_{3}
&=&c_{1}^{3}+2(4-c_{1}^{2})c_{1}x-c_{1}(4-c_{1}^{2})x^{2}+2(4-c_{1}^{2})%
\left( 1-\left\vert x\right\vert ^{2}\right) z, \\
4d_{3}
&=&d_{1}^{3}+2(4-d_{1}^{2})d_{1}y-d_{1}(4-d_{1}^{2})y^{2}+2(4-d_{1}^{2})%
\left( 1-\left\vert y\right\vert ^{2}\right) w,
\end{eqnarray*}%
\begin{equation}
c_{3}-d_{3}=\frac{c_{1}^{3}}{2}+\frac{c_{1}\left( 4-c_{1}^{2}\right) }{2}%
(x+y)-\frac{c_{1}\left( 4-c_{1}^{2}\right) }{2}(x^{2}+y^{2})+\frac{\left(
4-c_{1}^{2}\right) }{2}\left( \left( 1-\left\vert x\right\vert ^{2}\right)
z-\left( 1-\left\vert y\right\vert ^{2}\right) w\right) .  \label{eq17}
\end{equation}%
for some $x,y,$ $z,w$ with $\left\vert x\right\vert \leq 1,\left\vert
y\right\vert \leq 1,\left\vert z\right\vert \leq 1$ and $\left\vert
w\right\vert \leq 1.$Using (\ref{eq16}) and (\ref{eq17}) in (\ref{eq15}),
and applying the triangle inequality we have%
\begin{eqnarray*}
&&\left\vert a_{2}a_{4}-a_{3}^{2}\right\vert =\left\vert -\frac{1}{3}\left(
1-\beta \right) ^{4}c_{1}^{4}+\frac{1}{16}\left( 1-\beta \right)
^{3}c_{1}^{2}(4-c_{1}^{2})(x-y)\right. \\
&&+\frac{1}{6}\left( 1-\beta \right) ^{2}c_{1}\left[ \frac{c_{1}^{3}}{2}+%
\frac{(4-c_{1}^{2})c_{1}}{2}(x+y)-\frac{(4-c_{1}^{2})c_{1}}{4}(x^{2}+y^{2})+%
\frac{(4-c_{1}^{2})}{2}\left( (1-\left\vert x\right\vert
^{2})z-(1-\left\vert y\right\vert ^{2})w\right) \right] \\
&&\left. -\frac{1}{64}\left( 1-\beta \right)
^{2}(4-c_{1}^{2})^{2}(x-y)^{2}\right\vert \\
&\leq &\frac{1}{3}\left( 1-\beta \right) ^{4}c_{1}^{4}+\frac{1}{12}\left(
1-\beta \right) ^{2}c_{1}^{4}+\frac{1}{6}\left( 1-\beta \right)
^{2}c_{1}(4-c_{1}^{2}) \\
&&+\left[ \frac{1}{16}\left( 1-\beta \right) ^{3}c_{1}^{2}(4-c_{1}^{2})+%
\frac{1}{12}\left( 1-\beta \right) ^{2}c_{1}^{2}(4-c_{1}^{2})\right]
(\left\vert x\right\vert +\left\vert y\right\vert ) \\
&&+\left[ \frac{1}{24}\left( 1-\beta \right) ^{2}c_{1}^{2}(4-c_{1}^{2})-%
\frac{1}{12}\left( 1-\beta \right) ^{2}c_{1}(4-c_{1}^{2})\right] (\left\vert
x\right\vert ^{2}+\left\vert y\right\vert ^{2})+\frac{1}{64}\left( 1-\beta
\right) ^{2}(4-c_{1}^{2})^{2}(\left\vert x\right\vert +\left\vert
y\right\vert )^{2}.
\end{eqnarray*}

Since $p\in $ $\mathcal{P},$ so $\left\vert c_{1}\right\vert \leq 2.$
Letting $c_{1}=c,$ we may assume without restriction that $c\in \lbrack
0,2]. $ Thus, for $\lambda =\left\vert x\right\vert \leq 1$ and $\mu
=\left\vert y\right\vert \leq 1$ we obtain%
\begin{equation}
\left\vert a_{2}a_{4}-a_{3}^{2}\right\vert \leq T_{1}+T_{2}(\lambda +\mu
)+T_{3}(\lambda ^{2}+\mu ^{2})+T_{4}(\lambda +\mu )^{2}=F(\lambda ,\mu ) 
\notag
\end{equation}%
where%
\begin{eqnarray*}
T_{1} &=&T_{1}(c)=\frac{\left( 1-\beta \right) ^{2}}{12}\left[ \left(
1+4\left( 1-\beta \right) ^{2}\right) c^{4}-2c^{3}+8c\right] \geq 0,\text{ }
\\
T_{2} &=&T_{2}(c)=\frac{1}{48}\left( 1-\beta \right)
^{2}c^{2}(4-c^{2})(7-3\beta )\geq 0, \\
T_{3} &=&T_{3}(c)=\frac{1}{24}\left( 1-\beta \right) ^{2}c(4-c^{2})(c-2)\leq
0,\text{ } \\
T_{4} &=&T_{4}(c)=\frac{1}{64}\left( 1-\beta \right)
^{2}(4-c_{1}^{2})^{2}\geq 0.
\end{eqnarray*}

Now we need to maximize $F(\lambda ,\mu )$ in the closed square $\mathbb{S}%
=\left\{ (\lambda ,\mu ):0\leq \lambda \leq 1,0\leq \mu \leq 1\right\} .$
Since $T_{3}<0$ and $T_{3}+2T_{4}>0$ for $c\in \lbrack 0,2)$, we conclude
that%
\begin{equation*}
F_{\lambda \lambda }\cdot F_{\mu \mu }-\left( F_{\lambda \mu }\right) ^{2}<0.
\end{equation*}

Thus the function $F$ cannot have a local maximum in the interior of the
square $\mathbb{S}$. Now, we investigate the maximum of $F$ on the boundary
of the square $\mathbb{S}$.

For $\lambda =0$ and $0\leq \mu \leq 1$ $\left( \text{similarly }\mu =0\text{
and }0\leq \lambda \leq 1\right) ,$ we obtain%
\begin{equation*}
F(0,\mu )=G(\mu )=\left( T_{3}+T_{4}\right) \mu ^{2}+T_{2}\mu +T_{1}.
\end{equation*}

i. \textit{The case }$T_{3}+T_{4}\geq 0:$ In this case for $0<\mu <1$ and
any fixed $c$ with $0\leq c<2,$ it is clear that $G^{\prime }(\mu )=2\left(
T_{3}+T_{4}\right) \mu +T_{2}>0,$ that is, $G(\mu )$ is an increasing
function. Hence, for fixed $c\in \lbrack 0,2),$ the maximum of $G(\mu )$
occurs at $\mu =1,$ and%
\begin{equation*}
\max G(\mu )=G(1)=T_{1}+T_{2}+T_{3}+T_{4}.
\end{equation*}

ii. \textit{The case }$T_{3}+T_{4}<0:$ Since $T_{2}+2\left(
T_{3}+T_{4}\right) \geq 0$ for $0<\mu <1$ and any fixed $c$ with $0\leq c<2,$
it is clear that $T_{2}+2\left( T_{3}+T_{4}\right) <2\left(
T_{3}+T_{4}\right) \mu +T_{2}<T_{2}$ and so $G^{\prime }(\mu )>0.$ Hence for
fixed $c\in \lbrack 0,2),$ the maximum of $G(\mu )$ occurs at $\mu =1.$

Also for $c=2$ we obtain%
\begin{equation}
F(\lambda ,\mu )=\frac{4}{3}\left( 1-\beta \right) ^{2}(4\beta ^{2}-8\beta
+5).  \label{eq171}
\end{equation}

Taking into account the value (\ref{eq171}), and the cases \textit{i} and 
\textit{ii}, for $0\leq \mu \leq 1$ and any fixed $c$ with $0\leq c\leq 2$,%
\begin{equation*}
\max G(\mu )=G(1)=T_{1}+T_{2}+T_{3}+T_{4}.
\end{equation*}%
For $\lambda =1$ and $0\leq \mu \leq 1$ $\left( \text{similarly }\mu =1\text{
and }0\leq \lambda \leq 1\right) ,$ we obtain%
\begin{equation*}
F(1,\mu )=H(\mu )=\left( T_{3}+T_{4}\right) \mu ^{2}+\left(
T_{2}+2T_{4}\right) \mu +T_{1}+T_{2}+T_{3}+T_{4}.
\end{equation*}

Similarly to the above cases of $T_{3}+T_{4},$ we get that%
\begin{equation*}
\max H(\mu )=H(1)=T_{1}+2T_{2}+2T_{3}+4T_{4}.
\end{equation*}

Since $G(1)\leq H(1)$ for $c\in \lbrack 0,2],$ $\max F(\lambda ,\mu )=F(1,1)$
on the boundary of the square $\mathbb{S}$. Thus the maximum of $F$ occurs
at $\lambda =1$ and $\mu =1$ in the closed square $\mathbb{S}$.

Let $K:\left[ 0,2\right] \rightarrow 
\mathbb{R}
$ 
\begin{equation}
K(c)=\max F(\lambda ,\mu )=F(1,1)=T_{1}+2T_{2}+2T_{3}+4T_{4}.  \label{eq18}
\end{equation}

Substituting the values of $T_{1},T_{2},T_{3}$ and $T_{4}$ in the function $%
K $ defined by (\ref{eq18}), yield%
\begin{equation*}
K(c)=\frac{\left( 1-\beta \right) ^{2}}{48}\left[ \left( 16\beta
^{2}-26\beta +5\right) c^{4}+24(2-\beta )c^{2}+48\right] .
\end{equation*}

Assume that $K(c)$ has a maximum value in an interior of $c\in \lbrack 0,2]$%
, by elementary calculation we find 
\begin{equation}
K^{\prime }(c)=\frac{\left( 1-\beta \right) ^{2}}{12}\left[ \left( 16\beta
^{2}-26\beta +5\right) c^{3}+12(2-\beta )c\right] .  \label{eq21}
\end{equation}

As a result of some calculations we can do the following examine:

\textbf{Case 1: }Let\textbf{\ }$16\beta ^{2}-26\beta +5\geq 0,$ that is$,$ $%
\beta \in \left[ 0,\frac{13-\sqrt{89}}{16}\right] .$ Therefore $K^{\prime
}(c)>0$ for $c\in (0,2).$ Since $K$ is an increasing function in the
interval $(0,2)$, maximum point of $K$ must be on the boundary of $c\in
\lbrack 0,2],$ that is, $c=2$. Thus, we have%
\begin{equation*}
\underset{0\leq c\leq 2}{\max }K(c)=K(2)=\frac{4}{3}\left( 1-\beta \right)
^{2}\left( 4\beta ^{2}-8\beta +5\right) .
\end{equation*}

\textbf{Case 2: }Let\textbf{\ }$16\beta ^{2}-26\beta +5<0,$ that is$,$ $%
\beta \in \left( \frac{13-\sqrt{89}}{16},1\right) .$ Then $K^{\prime }(c)=0$
implies the real critical point $c_{0_{1}}=0$ or $c_{0_{2}}=\sqrt{\frac{%
-12(2-\beta )}{16\beta ^{2}-26\beta +5}}.$ When $\beta \in \left( \frac{13-%
\sqrt{89}}{16},\frac{29-\sqrt{137}}{32}\right] ,$ we observe that $%
c_{0_{2}}\geq 2$, that is, $c_{0_{2}}$ is out of the interval $(0,2)$.
Therefore the maximum value of $K(c)$ occurs at $c_{0_{1}}=0$ or $%
c=c_{0_{2}} $ which contradicts our assumption of having the maximum value
at the interior point of $c\in \lbrack 0,2]$. Since $K$ is an increasing
function in the interval $(0,2)$, maximum point of $K$ must be on the
boundary of $c\in \lbrack 0,2],$ that is, $c=2$. Thus, we have%
\begin{equation*}
\underset{0\leq c\leq 2}{\max }K(c)=K(2)=\frac{4}{3}\left( 1-\beta \right)
^{2}\left( 4\beta ^{2}-8\beta +5\right) .
\end{equation*}

When $\beta \in \left( \frac{29-\sqrt{137}}{32},1\right) $ we observe that $%
c_{0_{2}}<2$, that is, $c_{0_{2}}$ is interior of the interval $[0,2]$.
Since $K^{\prime \prime }(c_{0_{2}})<0,$ the maximum value of $K(c)$ occurs
at $c=c_{0_{2}}.$ Thus, we have 
\begin{equation*}
\underset{0\leq c\leq 2}{\max }K(c)=K(c_{0_{2}})=K\left( \sqrt{\frac{%
-12(2-\beta )}{16\beta ^{2}-26\beta +5}}\right) =\left( 1-\beta \right)
^{2}\left( \frac{13\beta ^{2}-14\beta -7}{16\beta ^{2}-26\beta +5}\right) .
\end{equation*}

This completes the proof of the Theorem \ref{t1}.
\end{proof}

For $\beta =0$, Theorem \ref{t1} readily yields the following coefficient
estimates for bi-starlike functions.

\begin{corollary}
\label{c1} Let $f(z)$ given by (\ref{eq1}) be in the class $\mathcal{S}%
_{\sigma }^{\ast }.$ Then%
\begin{equation*}
\left\vert a_{2}a_{4}-a_{3}^{2}\right\vert \leq \frac{20}{3}.
\end{equation*}
\end{corollary}

Our second main result for the class $\mathcal{K}_{\sigma }(\beta )$ is
following:

\begin{theorem}
\label{t2} Let $f(z)$ given by (\ref{eq1}) be in the class $\mathcal{K}%
_{\sigma }(\beta ),$ $0\leq \beta <1.$ Then%
\begin{equation}
\left\vert a_{2}a_{4}-a_{3}^{2}\right\vert \leq \frac{\left( 1-\beta \right)
^{2}}{24}\left( \frac{5\beta ^{2}+8\beta -32}{3\beta ^{2}-3\beta -4}\right)
\label{eq22}
\end{equation}
\end{theorem}

\begin{proof}
Let $f\in \mathcal{K}_{\sigma }(\beta )\;$and $g=f^{-1}.$\ Then%
\begin{equation}
1+\frac{zf^{\prime \prime }(z)}{f^{\prime }(z)}=\beta +(1-\beta )p(z)\;\text{%
and }1+\frac{wg^{\prime \prime }(w)}{g^{\prime }(w)}=\beta +(1-\beta )q(w)
\label{eq23}
\end{equation}%
where $p(z)=1+c_{1}z+c_{2}z^{2}+...$ and $q(w)=1+d_{1}w+d_{2}w^{2}+...$ in $%
\mathcal{P}$.

Now, equating the coefficients in (\ref{eq23}), we have%
\begin{eqnarray}
2a_{2} &=&(1-\beta )c_{1},  \label{eq25} \\
6a_{3}-4a_{2}^{2} &=&(1-\beta )c_{2},  \label{eq26} \\
12a_{4}-18a_{3}a_{2}+8a_{2}^{3} &=&(1-\beta )c_{3}  \label{eq27}
\end{eqnarray}%
and%
\begin{eqnarray}
-2a_{2} &=&(1-\beta )d_{1},  \label{eq28} \\
8a_{2}^{2}-6a_{3} &=&(1-\beta )d_{2},  \label{eq29} \\
-32a_{2}^{3}+42a_{3}a_{2}-12a_{4} &=&(1-\beta )d_{3}.  \label{eq30}
\end{eqnarray}%
From (\ref{eq25}) and (\ref{eq28}), we arrive at 
\begin{equation}
c_{1}=-d_{1}  \label{eq31}
\end{equation}%
and%
\begin{equation}
a_{2}=\frac{1}{2}(1-\beta )c_{1}.  \label{eq33}
\end{equation}%
Now, from (\ref{eq26}), (\ref{eq29}) and (\ref{eq33}), we get that%
\begin{equation}
a_{3}=\frac{1}{4}\left( 1-\beta \right) ^{2}c_{1}^{2}+\frac{1}{12}\left(
1-\beta \right) \left( c_{2}-d_{2}\right) .  \label{eq34}
\end{equation}%
Also, from (\ref{eq27}) and (\ref{eq30}), we find that 
\begin{equation}
a_{4}=\frac{5}{48}\left( 1-\beta \right) ^{3}c_{1}^{3}+\frac{5}{48}\left(
1-\beta \right) ^{2}c_{1}\left( c_{2}-d_{2}\right) +\frac{1}{24}(1-\beta
)\left( c_{3}-d_{3}\right) .  \label{eq35}
\end{equation}%
Thus, we can easily establish that%
\begin{eqnarray}
&&\left\vert a_{2}a_{4}-a_{3}^{2}\right\vert =\left\vert -\frac{1}{96}\left(
1-\beta \right) ^{4}c_{1}^{4}+\frac{1}{96}\left( 1-\beta \right)
^{3}c_{1}^{2}\left( c_{2}-d_{2}\right) \right.  \notag \\
&&\text{ \ \ \ \ \ \ \ \ \ \ \ \ \ \ }\left. +\frac{1}{48}\left( 1-\beta
\right) ^{2}c_{1}\left( c_{3}-d_{3}\right) -\frac{1}{144}\left( 1-\beta
\right) ^{2}\left( c_{2}-d_{2}\right) ^{2}\right\vert .  \label{eq36}
\end{eqnarray}

Using (\ref{eq16}) and (\ref{eq17}) in (\ref{eq36}), we have%
\begin{eqnarray*}
&&\left\vert a_{2}a_{4}-a_{3}^{2}\right\vert =\left\vert -\frac{1}{96}\left(
1-\beta \right) ^{4}c_{1}^{4}+\frac{1}{192}\left( 1-\beta \right)
^{3}c_{1}^{2}(4-c_{1}^{2})(x-y)\right. \\
&&+\frac{1}{48}\left( 1-\beta \right) ^{2}c_{1}\left[ \frac{c_{1}^{3}}{2}+%
\frac{(4-c_{1}^{2})c_{1}}{2}(x+y)-\frac{(4-c_{1}^{2})c_{1}}{4}(x^{2}+y^{2})+%
\frac{(4-c_{1}^{2})}{2}\left( (1-\left\vert x\right\vert
^{2})z-(1-\left\vert y\right\vert ^{2})w\right) \right] \\
&&\left. -\frac{1}{288}\left( 1-\beta \right)
^{2}(4-c_{1}^{2})^{2}(x-y)^{2}\right\vert \\
&\leq &\frac{1}{96}\left( 1-\beta \right) ^{4}c_{1}^{4}+\frac{1}{96}\left(
1-\beta \right) ^{2}c_{1}^{4}+\frac{1}{48}\left( 1-\beta \right)
^{2}c_{1}(4-c_{1}^{2}) \\
&&+\left[ \frac{1}{192}\left( 1-\beta \right) ^{3}c_{1}^{2}(4-c_{1}^{2})+%
\frac{1}{96}\left( 1-\beta \right) ^{2}c_{1}^{2}(4-c_{1}^{2})\right]
(\left\vert x\right\vert +\left\vert y\right\vert ) \\
&&+\left[ \frac{1}{192}\left( 1-\beta \right) ^{2}c_{1}^{2}(4-c_{1}^{2})-%
\frac{1}{96}\left( 1-\beta \right) ^{2}c_{1}(4-c_{1}^{2})\right] (\left\vert
x\right\vert ^{2}+\left\vert y\right\vert ^{2})+\frac{1}{576}\left( 1-\beta
\right) ^{2}(4-c_{1}^{2})^{2}(\left\vert x\right\vert +\left\vert
y\right\vert )^{2}.
\end{eqnarray*}%
Since $p\in $ $\mathcal{P},$ so $\left\vert c_{1}\right\vert \leq 2.$ Taking 
$c_{1}=c,$ we may assume without restriction that $c\in \lbrack 0,2].$ Thus,
for $\lambda =\left\vert x\right\vert \leq 1$ and $\mu =\left\vert
y\right\vert \leq 1$ we obtain%
\begin{equation}
\left\vert a_{2}a_{4}-a_{3}^{2}\right\vert \leq M_{1}+M_{2}(\lambda +\mu
)+M_{3}(\lambda ^{2}+\mu ^{2})+M_{4}(\lambda +\mu )^{2}=\Psi (\lambda ,\mu )
\notag
\end{equation}%
where%
\begin{eqnarray*}
M_{1} &=&M_{1}(c)=\frac{\left( 1-\beta \right) ^{2}}{96}\left[ \left(
1+\left( 1-\beta \right) ^{2}\right) c^{4}-2c^{3}+8c\right] \geq 0,\text{ }
\\
M_{2} &=&M_{2}(c)=\frac{1}{192}\left( 1-\beta \right)
^{2}c^{2}(4-c^{2})(3-\beta )\geq 0, \\
M_{3} &=&M_{3}(c)=\frac{1}{192}\left( 1-\beta \right)
^{2}c(4-c^{2})(c-2)\leq 0,\text{ } \\
M_{4} &=&M_{4}(c)=\frac{1}{576}\left( 1-\beta \right)
^{2}(4-c_{1}^{2})^{2}\geq 0.
\end{eqnarray*}

Therefore we need to maximize $\Psi (\lambda ,\mu )$ in the closed square $%
\mathbb{S}=\left\{ (\lambda ,\mu ):0\leq \lambda \leq 1,0\leq \mu \leq
1\right\} .$ To show that the maximum of $\Psi $\ we can follow the maximum
of $F$ in the Theorem \ref{t1}. Thus the maximum of $\Psi $ occurs at $%
\lambda =1$ and $\mu =1$ in the closed square $\mathbb{S}$. Let $\Phi :\left[
0,2\right] \rightarrow 
\mathbb{R}
$ defined by 
\begin{equation}
\Phi (c)=\max \Psi (\lambda ,\mu )=\Psi (1,1)=M_{1}+2M_{2}+2M_{3}+4M_{4}.
\label{eq37}
\end{equation}

Substituting the values of $M_{1},M_{2},M_{3}$ and $M_{4}$ in the function $%
\Phi $ given by (\ref{eq37}), yield%
\begin{equation*}
\Phi (c)=\frac{\left( 1-\beta \right) ^{2}}{288}\left[ \left( 3\beta
^{2}-3\beta -4\right) c^{4}+4(8-3\beta )c^{2}+32\right] .
\end{equation*}

Assume that $\Phi (c)$ has a maximum value in an interior of $c\in \lbrack
0,2]$, by elementary calculation we find 
\begin{equation*}
\Phi ^{\prime }(c)=\frac{\left( 1-\beta \right) ^{2}}{72}\left[ \left(
3\beta ^{2}-3\beta -4\right) c^{3}+2(8-3\beta )c\right] .
\end{equation*}

Setting $\Phi ^{\prime }(c)=0,$ since $0<c<2,$ and $3\beta ^{2}-3\beta -4<0$
and $8-3\beta >0$ for every $\beta \in \left[ 0,1\right) $ we have the real
critical poin $c_{0_{3}}=\sqrt{\frac{2(3\beta -8)}{3\beta ^{2}-3\beta -4}}.$
Since $c_{0_{3}}\leq 2$ for every $\beta \in \left[ 0,1\right) $ and so $%
\Phi ^{\prime \prime }(c_{0_{3}})<0$, the maximum value of $\Phi (c)$
corresponds to $c=c_{0_{3}},$ that is,%
\begin{equation*}
\underset{0<c<2}{\max }\Phi (c)=\Phi (c_{0_{3}})=\Phi \left( \sqrt{\frac{%
2(3\beta -8)}{3\beta ^{2}-3\beta -4}}\right) =\frac{\left( 1-\beta \right)
^{2}}{24}\left( \frac{5\beta ^{2}+8\beta -32}{3\beta ^{2}-3\beta -4}\right) .
\end{equation*}%
On the other hand, 
\begin{equation*}
\Phi \left( 0\right) =\frac{\left( 1-\beta \right) ^{2}}{9}\text{ and }\Phi
\left( 2\right) =\frac{\left( 1-\beta \right) ^{2}}{6}\left( \beta
^{2}-2\beta +2\right) .
\end{equation*}%
Consequently, since $\Phi \left( 0\right) <\Phi \left( 2\right) \leq \Phi
(c_{0_{3}})$ we obtain$\underset{0\leq c<\leq 2}{\max }\Phi (c)=\Phi
(c_{0_{3}})$.

This completes the proof of the Theorem \ref{t2}.
\end{proof}

For $\beta =0$, Theorem \ref{t2} readily yields the following coefficient
estimates for bi-convex functions.

\begin{corollary}
\label{c2}Let $f(z)$ given by (\ref{eq1}) be in the class $\mathcal{K}%
_{\sigma }.$ Then%
\begin{equation*}
\left\vert a_{2}a_{4}-a_{3}^{2}\right\vert \leq \frac{1}{3}.
\end{equation*}
\end{corollary}

\end{document}